\documentclass[12pt,a4paper]{amsart}
\usepackage[T1]{fontenc}
\usepackage[utf8]{inputenc}
\usepackage[vmargin=80pt,hmargin=60pt]{geometry}
\usepackage[all]{xy}
\newtheorem{theorem}{Theorem}[section]

\newtheorem{lemma}[theorem]{Lemma}
\newtheorem{proposition}[theorem]{Proposition}
\theoremstyle{definition}
\newtheorem{definition}[theorem]{Definition}
\theoremstyle{remark}
\newtheorem{example}[theorem]{Example}
\newtheorem{remark}[theorem]{Remark}
\newtheorem{note}[theorem]{Note}
\numberwithin{equation}{section}

\newcommand{\prth}[1]{\left(#1\right)}
\newcommand{\lie}[1]{\left[#1\right]}

\newcommand{\set}[1]{\left\{#1\right\}}
\newcommand{\NN}{\mathbb N}
\newcommand{\RR}{\mathbb R}
\newcommand{\R}{\mathbb R}

\newcommand{\tensor}{\otimes}
\newcommand{\To}{\longrightarrow}
\newcommand{\dd}{d}
\newcommand{\dx}{\dd x}
\newcommand{\dmx}{\dd^mx}
\newcommand{\dmxi}{\dd^{m-1}x_i}
\newcommand{\dy}{\dd y}
\newcommand{\du}{\dd u}
\newcommand{\hh}{\mathbf{h}}
\newcommand{\HH}{\mathbf{H}}

\newcommand{\pr}{\mathit{pr}}
\renewcommand{\and}{\quad\textrm{and}\quad}
\newcommand{\cf}{\emph{cf.}}

\newcommand{\etal}{\emph{et~al.}}

\newcommand{\secref}[1]{\S\ref{#1}}

\begin{document}
\title{Unambiguous formalism for higher-order Lagrangian field theories}
\author[C. M. Campos]{Cédric M. Campos}
\address{Instituto de Ciencias Matem\'aticas\\CSIC-UAM-UC3M-UCM\\Serrano 123\\28006 Madrid\\Madrid (Spain)}
\email[C. M. Campos]{cedricmc@imaff.cfmac.csic.es}
\author[M. de León]{Manuel de León}
\email[M. de León]{mdeleon@imaff.cfmac.csic.es}
\author[D. Martín de Diego]{David Martín de Diego}
\email[D. Martín de Diego]{d.martin@imaff.cfmac.csic.es}
\author[J. Vankerschaver]{Joris Vankerschaver}
\address{Control and Dynamical Systems, California Institute of Technology\\California (USA) \\ Department of Mathematical Physics and Astronomy \\ Ghent University \\ Krijgslaan 281, B-9000 Ghent (Belgium)}
\email[J. Vankerschaver]{jv@caltech.edu}
\thanks{Plublished at J. Phys. A: Math. Theor. 42 (2009) 475207, doi:10.1088/1751-8113/42/47/475207}
\date{November 9th, 2009}
\begin{abstract}
The aim of this paper is to propose an unambiguous intrinsic formalism for higher-order field theories which avoids the arbitrariness in the generalization of the conventional description of field theories, and implies the existence of different Cartan forms and Legendre transformations. We propose a differential-geometric setting for the dynamics of a higher-order field theory, based on the Skinner and Rusk formalism for mechanics. This approach incorporates aspects of both, the Lagrangian and the Hamiltonian description, since the field equations are formulated using the Lagrangian on a higher-order jet bundle and the canonical multisymplectic form on its affine dual. As both of these objects are uniquely defined, the Skinner-Rusk approach has the advantage that it does not suffer from the arbitrariness in conventional descriptions. The result is that we obtain a unique and  global intrinsic version of the Euler-Lagrange equations for higher-order field theories. Several examples illustrate our construction.
\end{abstract}
\maketitle

\section{Introduction} \label{sec.introduction}
During the last decades of the past century, there have been different studies and attempts to define in a global and intrinsic way the higher-order calculus of variations in several independent variables. The standard geometric picture starts with a Lagrangian function  $L: J^k\pi\to \RR$
where $J^k\pi$ is the $k$th-order jet bundle of a fiber bundle $\pi: E\rightarrow M$. The main objectives are to describe the associated Euler-Lagrange equations for sections of the fiber bundle, to derive Cartan forms for use in intrinsic versions of the above equations, and to construct adequate Legendre maps which permit to write the equations in the Hamiltonian side (see, for instance,  \cite{AldAzc78a,AldAzc80a,Ddck84a,Ddck84b,GrcMnz90,Gty91,HrkKlr83,Krpk83,LeonRdrs87b,LeonRdrs87a,SndCrmp90} for further information).

The situation is well established for the case of one independent variable (higher order mechanics) and for the case of first order calculus of variations \cite{Grc73,GchtMngrSrdn97,GldStrn73,Hrmn73}. In this last situation, the typical expression  of the Cartan form associated in classical mechanics to a Lagrangian $L: J^1\pi\to \RR$ may be written as $S^*(dL)+Ldt$, where $S^*$ is the adjoint of the vertical endomorphism acting on 1-forms. In order to generalize this concept to higher order field theories, one needs to define a mapping from 1-forms (the differential of $L$) to $m$-forms and to incorporate in a global way the higher order derivatives.
This is one of the reasons for the  degree of arbitrariness in the definition of Cartan forms for Lagrangian functions $L: J^k\pi\to \RR$, if $k>1$ and $\dim M>1$. In other words, there will be different Cartan forms which carry out the same function in order to define an intrinsic formulation of Euler-Lagrange equations. The main reason of this problem is the commutativity of repeated partial differentiation. Therefore, the Cartan form is unique if (and only if) either $k$ or $m$  equals one.

In the literature, we find different approaches to fix the Cartan form for higher order field theories. A direct attempt is the approach by Aldaya and Azc\'arraga \cite{AldAzc78a,AldAzc80a}. Another point of view is that by Arens \cite{Arns81}, which consists of injecting the jet bundle $J^k\pi$ to an appropriate first-order jet bundle by the introduction of a great number of variables into the theory and Lagrange multipliers. From a more geometrical point of view, Garc\'ia and Mu\~noz described a method of constructing global Poincar\'e-Cartan forms in the higher order calculus of variations in fibered spaces by means of a linear connections (see \cite{GrcMnz82,GrcMnz90}).  In particular they show that the Cartan forms depend on the choice of two connections, a linear connection on the base $M$ and  a linear connection on the vertical bundle $V\pi$. Later, Crampin and Saunders \cite{SndCrmp90} proposed the use of an operator analogous to the almost tangent structure canonically defined on the tangent bundle of a given configuration manifold $M$ for the construction of global Poincar\'e-Cartan forms; this operator depends on the chosen volume form on the base.

In our paper, we propose an alternative way, avoiding the use of additional structures, working only with intrinsic objects from both the Lagrangian and Hamiltonian sides. This formalism is strongly based on the one developed by Skinner and Rusk \cite{Skn83,SknRsk83a,SknRsk83b}. In order to deal with singular Lagrangian systems, Skinner and Rusk construct a Hamiltonian system on the Whitney sum $TQ \oplus T^\ast Q$ of the tangent and cotangent bundles of the configuration manifold $Q$. The advantage of their approach lies on the fact that the second order condition of the dynamics is automatically satisfied. This does not happen in the Lagrangian side of the Gotay and Nester formulation, where the second-order condition problem has to be considered after the implementation of the constraint algorithm (see \cite{GtyNstr79,GtyNstr80,GtyNstrHnds78}), besides other formalisms which include the second-order condition from the very beginning (see \cite{Crny90,CrnyLpzRmn87}).

For higher-order field theories, we start with a Lagrangian function defined on $J^k\pi$. We consider the fibration $\pi_{W_0,M}:W_0\To M$, where $W_0=J^k\pi\times_{J^{k-1}\pi}\Lambda^m_2(J^{k-1}\pi)$ is a fibered product, the velocity-momentum space. On $W_0$ we construct a (pre-)multisymplectic form by pulling back the canonical multisymplectic form on $\Lambda^m_2(J^{k-1}\pi)$, and we define a convenient Hamiltonian from a natural canonical pairing and the given Lagrangian function. The solutions of the field equations are viewed as integral sections of Ehresmann connections in the fibration $\pi_{W_0,M}:W_0\To M$. In this space we obtain a global, intrinsic  and unique expression for a Cartan type equation for the Euler-Lagrange equations for higher-order field theories. Additionally, we obtain a resultant constraint algorithm. Our scheme is applied to several examples to illustrate our method.

Apart from the lack of ambiguity inherent in our construction, we emphasize that our formalism can be easily extended to the case of higher-order field theories with constraints and optimal control problems for partial differential equations. In this way, we obtain a unified, geometric description of both types of systems, with possible future applications in the theory of symmetry reduction and the construction of numerical methods preserving geometric structure (see \cite{LmkRch04}). This will be the topic of future research.

While finalizing this paper, we found out about the work of L. Vitagliano \cite{Vtgl09} who independently used the unified formalism framework to study higher-order field theories, using techniques from secondary calculus.

Throughout the paper, lower case Latin (resp. Greek) letters will usually denote indexes that range between $1$ and $m$ (resp. $1$ and $n$). Capital Latin letters will usually denote multi-indexes whose length ranges between $0$ and $k$. In particular, in section \secref{sec.jet.bundles} and all later sections, $I$ and $J$ will usually denote multi-indexes whose length goes from $0$ to $k-1$ and $0$ to $k$, respectively; and $K$ will denote multi-indexes whose length is equal to $k$. The Einstein notation for repeated indexes and multi-indexes is used but, for clarity, in some cases the summation for multi-indexes will be indicated.

\section{Jet Bundles} \label{sec.jet.bundles}
Let $(E,\pi,M)$ be a fiber bundle whose base space $M$ is an orientable differentiable manifold of dimension $m$, and whose fibers have dimension $n$, thus $E$ is $(m+n)$-dimensional. Adapted coordinated systems will be of the form $(x^i,u^\alpha)$, where $(x^i)$ is a local coordinate system in $M$ and $(u^\alpha)$ denotes fiber coordinates. We fix a volume form $\eta$ on the base manifold $M$. For a compatible chart $(x^i)$ with respect to the volume form, $\eta$ is written $\dmx=\dx^1\wedge\dots\wedge\dx^m$, and we will write $\dmxi$ for the contraction $i_{\partial/\partial x^i}\dmx$ ($\dd^{m-2}x_{ij}=i_{\partial/\partial x^j}\dmxi$ and so on).

Given a point $p\in M$, let $\phi,\psi:M\To E$ be two smooth local sections around $p$. We say that $\phi$ and $\psi$ are \emph{$k$-equivalent} at $p$ (with $k\geq1$) if the sections and all their partial derivatives until order $k$ coincide at $p$, that is, if
\[ \phi(p)=\psi(p)\textrm{ and } \left.\frac{\partial^k\phi^\alpha}{\partial x^{i_1}\cdots\partial x^{i_k}}\right|_p=\left.\frac{\partial^k\psi^\alpha}{\partial x^{i_1}\cdots\partial x^{i_k}}\right|_p, \]
for all $1\leq\alpha\leq n$, $1\leq i_j\leq m$, $1\leq j\leq k$. Note that this is independent of the chosen coordinate system (adapted or not) and, therefore, to be $k$-equivalent is an equivalence relation (see \cite{LeonRdrs85,LeonRdrs89,Snd89}, for more details).

\begin{definition} \label{def.jet}
Let $(E,\pi,M)$ be a fiber bundle and $p\in M$. Given a smooth local section $\phi\in\Gamma_p(\pi)$, the equivalence class of $k$-equivalent smooth local sections (with $k\geq1$) around $p$ that contains $\phi$ is called the \emph{$k$-jet of $\phi$ at $p$} and is denoted $j^k_p\phi$. The set of all the $k$-jets of local sections, that is,
\[ \set{j^k_p\phi\ :\ p\in M, \phi\in\Gamma_p(\pi)}, \]
is called the \emph{$k$-th jet manifold of $\pi$} and denoted $J^k\pi$.
\end{definition}

These sets have interesting structures and relations between them, but before we present them, we will introduce a particular multi-index notation.

\begin{note}[The multi-index notation, \cite{Snd89}]
Given a function $f:\RR^m\To\RR$, its partial derivatives are denoted by
\[ f_{i_1i_2\cdots i_k} = \frac{\partial^kf}{\partial x_{i_1}\partial x_{i_2}\cdots\partial x_{i_k}}. \]
Since all the functions that we consider are smooth enough, their crossed derivatives coincide. Thus, the order in which the derivatives are taken is not important, but the number of times with respect to each variable.

Another notation to denote partial derivatives is defined through multi-indexes. A multi-index $I$ will be an $m$-tuple of non-negative integers. The $i$-th component of $I$ is denoted $I(i)$. Addition and substraction of multi-indexes are defined componentwise (whenever the result is still a multi-index), $(I\pm J)(i)=I(i)\pm J(i)$. The length of $I$ is the sum $|I|=\sum_iI(i)$, and its factorial $I!=\Pi_iI(i)!$. In particular, $1_i$ will be the multi-index that is zero everywhere except at the $i$-th component which is equal to 1.

Keeping in mind the above notations, we will denote the partial derivatives of a function $f:\RR^m\To\RR$ by:
\[ f_I = \frac{\partial^{|I|}f}{\partial x^I} = \frac{\partial^{I(1)+I(2)+\dots+I(m)}f}{\partial x_1^{I(1)}\partial x_2^{I(2)}\cdots\partial x_m^{I(m)}}. \]
Thus, given a multi-index $I$, $I(i)$ denotes the number of times the function is differentiated with respect to the $i$-th component. The former notation should not be confused with the latter one. For instance, the third order partial derivative $\frac{\partial^3f}{\partial x_2\partial x_3\partial x_2}$ (with $f:\RR^3\To\RR$) is denoted $f_{232}$ and $f_{(0,2,1)}$ respectively.
\end{note}

Let $(E,\pi,M)$ be a fiber bundle as before. An adapted coordinate system $(x^i,u^\alpha)$ on the total space $E$ induces adapted coordinates $(x^i,u^\alpha_I)$ (with $0\leq|I|\leq k$) on the $k$-jet manifold $J^k\pi$ given by:
\[ u^\alpha_I(j^k_p\phi) = \left.\frac{\partial^{|I|}\phi^\alpha}{\partial x^I}\right|_p, \]
from where we deduce that the dimension of $J^k\pi$ is
\[ \dim J^k\pi = m + n\cdot\sum_{l=0}^k\binom{m-1+l}{m-1}. \]
It is readily seen that $(J^k\pi,\pi_k,M)$ is a fiber bundle, where
\[ \pi_k(j^k_p\phi) = p  \quad  (\textrm{in coordinates }\pi_k(x^i,u^\alpha_I) = (x^i)). \]
Note that any local section $\phi\in\Gamma_p(\pi)$ can be lifted to a local section in $\Gamma_p(\pi_k)$ defining its lift by (see Diagram \ref{fig.jet.chain}):
\[ (j^k\phi)(p) = j^k_p\phi. \]
It is also clear that the $k$-jet manifold $J^k\pi$ fibers over the lower order $l$-jet manifolds $J^l\pi$ (see Diagram \ref{fig.jet.chain}), with $0\leq l<k$, where by convention $J^0\pi=E$ and where the projections are given by:
\[ \pi_{k,l}(j^k_p\phi) = j^l_p\phi  \quad  (\textrm{in coordinates }\pi_{k,l}(x^i,u^\alpha_I) = (x^i,u^\alpha_J), \textrm{ with }0\leq|I|\leq k,0\leq|J|\leq l). \]
\begin{figure}[h]
\[\xymatrix{
  E \ar[dd]^\pi & J^1\pi \ar[l]_{\pi_{1,0}} \ar[ddl]^{\pi_1} & J^2\pi \ar[l]_{\pi_{2,1}} \ar[ddll]^{\pi_2} & \cdots \ar[l] & J^k\pi \ar[l] \ar[ddllll]^{\pi_k} \\ \\
  M \ar@/^1.2pc/[uu]^\phi \ar@/_2pc/[uurrrr]_{j^k\phi}
  }\]
\caption{Chain of jets} \label{fig.jet.chain}
\end{figure}
In particular, $(J^k\pi,\pi_{k,k-1},J^{k-1}\pi)$ is an affine fiber bundle (see Cariñena \etal\ \cite{CrnyCrmpIbrt91} for the case $k=1$, or Saunders \cite{Snd89} for the general case), which is modeled on the vector bundle
\[ \pi_{k-1}^*(S^kT^*M)\otimes\pi_{k-1,0}^*(V\pi), \]
where $S^kT^*M$ is the space of symmetric $k$-tensors on $M$ and $V\pi$ is the vertical fiber bundle on $\pi$. Thus, taking repeated jets, $(J^1\pi_k,(\pi_k)_{1,0},J^k\pi)$ is also an affine fiber bundle. Furthermore, $J^{k+1}\pi$ can be naturally embedded into $J^1\pi_k$ (see Diagram \ref{fig.iterated.jet}). The inclusion map $i_{1,k}:J^{k+1}\pi\hookrightarrow J^1\pi_k$ is given by
\[ i_{1,k}(j^{k+1}_p\phi) = j^1_p(j^k\phi). \]
If we consider fiber coordinates $(x^i,u^\alpha_I,u^\alpha_{I;i})$ on $J^1\pi_k$ (with $0\leq|I|\leq k$), then $i_{1,k}(J^{k+1}\pi)$ is given by the equations
\[ \left\{\begin{array}{ll}u^\alpha_{I;i} = u^\alpha_{I+1_i},& \textrm{for } 0\leq|I|\leq k-1;\textrm{ and}\\ u^\alpha_{I;i} = u^\alpha_{J;j},& \textrm{when } |I|=|J|=k \textrm{ and } I+1_i=J+1_j.\end{array}\right. \]
\begin{figure}[h]
\[\xymatrix{
                                                      & J^{k+1}\pi \ar[d]^{\pi_{k+1,k}} \ar[dl]_{i_{1,k}} \\
  J^1\pi_k \ar[r]^{\ (\pi_k)_{1,0}} \ar[dr]_{(\pi_k)_1} & J^k\pi     \ar[d]^{\pi_k} \\
                                                      & M
  }\]
\caption{Iterated jet} \label{fig.iterated.jet}
\end{figure}

As we have said, $\pi_{k,k-1}:J^k\pi\To J^{k-1}\pi$ is an affine bundle, so we may consider its extended dual affine bundle $\pi_{k,k-1}^\dag:J^k\pi^\dag\To J^{k-1}\pi$ and its dual affine bundle $\pi_{k,k-1}^*:J^k\pi^*\To J^{k-1}\pi$. The extended dual bundle $(J^k\pi^\dag,\pi_{k,k-1}^\dag,J^{k-1}\pi)$ is a fiber bundle whose fibers consist of affine maps of the corresponding fibers of the affine bundle $(J^k\pi,\pi_{k,k-1}, J^{k-1}\pi)$. In its turn, $(J^k\pi^*,\pi_{k,k-1}^*,J^{k-1}\pi)$ is a fiber bundle whose fibers consist of classes of affine maps of the corresponding fibers of the affine bundle $(J^k\pi,\pi_{k,k-1}, J^{k-1}\pi)$, which differ by a constant. It can be shown that there exist canonical isomorphisms such that $J^k\pi^\dag\approx\Lambda^m_2(J^{k-1}\pi)$ and $J^k\pi^*\approx\Lambda^m_2(J^{k-1}\pi)/\Lambda^m_1(J^{k-1}\pi)$, where $\Lambda^m_r(J^{k-1}\pi)$ is the bundle of those $m$-forms over $J^{k-1}\pi$ that are annihilated when $r$ of their arguments are vertical with respect to $\pi_{k-1}:J^{k-1}\pi\To M$. Locally, the elements of $\Lambda^m_2(J^{k-1}\pi)$ are of the form
\[ p\,\dmx + p^{I,i}_\alpha\du^\alpha_I\wedge\dmxi, \]
where $0\leq|I|\leq k-1$. Thus, adapted coordinates on $J^{k-1}\pi$ induce coordinates on $\Lambda^m_2(J^{k-1}\pi)$ and $\Lambda^m_2(J^{k-1}\pi)/\Lambda^m_1(J^{k-1}\pi)$ of the form
\[ (x^i,u^\alpha_I,p,p^{I,i}_\alpha) \textrm{ and } (x^i,u^\alpha_I,p^{I,i}_\alpha), \]
respectively.

While $J^k\pi^\dag$ is naturally paired with $J^k\pi$, $\Lambda^m_2(J^{k-1}\pi)$ has a canonical multisymplectic form (see \cite{CntrIbrtLeon96,CntrIbrtLeon99,CrnyCrmpIbrt91}).
The pairing between the elements of $J^k\pi$ and $\Lambda^m_2(J^{k-1}\pi)$ is given by
\begin{equation} \label{eq.natural.pairing}
\Phi(j^k_x\phi,\omega_{j^{k-1}_x\phi})=a(x), \textrm{ such that } a(x)\eta(x)=(j^{k-1}\phi
)^*\omega_{j^{k-1}_x\phi};
\end{equation}
which is written in adapted coordinates
\begin{equation} \label{eq.natural.pairing.coordinates}
\Phi(x^i,u^\alpha_I,u^\alpha_K,p^{I,i}_\alpha,p) = p^{I,i}_\alpha u^\alpha_{I+1_i}+p,
\end{equation}
where $|I|=0,\dots,k-1$ and $|K|=k$.
The canonical multisymplectic $(m+1)$-form on $\Lambda^m_2(J^{k-1}\pi)$, which will be denoted $\Omega$, is written in coordinates
\begin{equation} \label{eq.canonical.form}
\Omega = -\dd p\wedge\dmx-\dd p^{I,i}_\alpha\wedge\du^\alpha_I\wedge\dmxi.
\end{equation}


\section[The Skinner-Rusk formalism]{The Skinner-Rusk formalism} \label{sec.s-r.formalism}
The generalization of the Skinner-Rusk formalism to higher order classical field theories will take place in the fibered product
\begin{equation} \label{eq.mixed.space}
W_0=J^k\pi\times_{J^{k-1}\pi}\Lambda^m_2(J^{k-1}\pi).
\end{equation}
The first order case is covered in \cite{LeonMrrMrt03,MrnMnzRmn03}; see also \cite{SknRsk83a,SknRsk83b} for the original treatment by Skinner and Rusk. The projection on the $i$-th factor will be denoted $\pr_i$ (with $i=1,2$) and the projection as fiber bundle over $J^{k-1}\pi$ will be $\pi_{W_0,J^{k-1}\pi}=\pi_{k,k-1}\circ\pr_1$ (see Diagram \ref{fig.dynamical.framework}). On $W_0$, adapted coordinate systems are of the form $(x^i,u^\alpha_I,u^\alpha_K,p^{I,i}_\alpha,p)$, where $|I|=0,\dots,k-1$ and $|K|=k$.

\begin{figure}[h]
\[\xymatrix{
  && W_0 \ar[lld]_{\pr_1} \ar[dd]^{\pi_{W_0,J^{k-1}\pi}} \ar[rrd]^{\pr_2}&&\\
  J^k\pi \ar[rrd]^{\pi_{k,k-1}} && && \Lambda^m_2(J^{k-1}\pi) \ar[lld] \\
  && J^{k-1}\pi \ar[d]^{\pi_{k-1}}&&\\
  && M &&
  }\]
\caption{The Skinner-Rusk framework} \label{fig.dynamical.framework}
\end{figure}

Assume that $L:J^k\pi\To\RR$ is a Lagrangian function. Together with the pairing $\Phi$ (equations \eqref{eq.natural.pairing} and \eqref{eq.natural.pairing.coordinates}), we use this Lagrangian $L$ to define a dynamical function $H_0$ (corresponding to the Hamiltonian) on $W_0$:
\begin{equation} \label{eq.hamiltonian}
H_0 = \Phi - L\circ\pr_1.
\end{equation}

Consider the canonical multisymplectic $(m+1)$-form $\Omega$ on $\Lambda^m_2(J^{k-1}\pi)$ (equation \eqref{eq.canonical.form}), whose pullback to $W_0$ shall be denoted also by $\Omega$. We define on $W_0$ the $(m+1)$-form
\begin{equation} \label{eq.hamiltonian.form}
\Omega_{H_0} = \Omega + \dd H_0\wedge\eta.
\end{equation}
In adapted coordinates
\begin{eqnarray}
H_0&=&p^{I,i}_\alpha u^\alpha_{I+1_i}+p-L(x^i,u^\alpha_I,u^\alpha_K)\\
\label{eq.hamiltonian.form.coordinates}
\Omega_{H_0} &=& -\dd p^{I,i}_\alpha\wedge\du^\alpha_I\wedge\dmxi+\prth{p^{I,i}_\alpha\du^\alpha_{I+1_i}+u^\alpha_{I+1_i}\dd p^{I,i}_\alpha-\frac{\partial L}{\partial u^\alpha_J}\du^\alpha_J}\wedge\dmx,
\end{eqnarray}
where $|I|=0,\dots,k-1$ and $|J|=0,\dots,k$.

\subsection{The dynamical equation} \label{sec.srf.dynamical.equation}
We search for an Ehresmann connection $\Gamma$ in the fiber bundle $\pi_{W_0,M}:W_0\To M$ whose horizontal projector be a solution of the \emph{dynamical equation} (see Appendix \ref{sec.app.connections}):
\begin{equation} \label{eq.dynamical.equation}
i_\hh\Omega_{H_0}=(m-1)\Omega_{H_0}.
\end{equation}
We will show that such a solution does not exist on the whole $W_0$. Thus, we need to restrict to the space on where such a solution exists, that is on
\begin{equation} \label{eq.w1.def}
\begin{array}{rcl}
W_1 &=& \{ w\in W_0\ /\ \exists\hh_w:T_wW_0\To T_wW_0\textrm{ linear such that }\hh_w^2=\hh_w,\\
      && \ker\hh_w=(V\pi_{W_0,M})_w,\ i_{\hh_w}\Omega_{H_0}(w)=(m-1)\Omega_{H_0}(w) \}.
\end{array}
\end{equation}

\begin{remark}
Equation \eqref{eq.dynamical.equation} is a generalization of equations that usually appear in first order field theories. In this particular case, from a given  Lagrangian function $L:J^1\pi\to\RR$ we may construct a unique  $(n+1)$-form $\Omega_L$ (the Poincar\'e-Cartan (n+1)-form). Hence, we have a geometrical characterization of the Euler-Lagrange equations for $L$ as follows. Let $\Gamma$ be an Ehresmann connection in $\pi_{1,0}:J^1\pi\to M$,  with horizontal projector $\hh$. Consider the equation
\begin{equation} \label{eq.dynamical.equation.rmk}
i_\hh\Omega_L = (n-1)\Omega_L .
\end{equation}
If $\hh$ has locally the from
\[ \hh\prth{\frac{\partial}{\partial x^i}} = \frac{\partial}{\partial x^i} + A^\alpha_i \frac{\partial}{\partial u^\alpha} + A^\alpha_{ji} \frac{\partial}{\partial u^\alpha_j}, \]
then a direct computation shows that equation \eqref{eq.dynamical.equation.rmk} holds if and only if
\begin{eqnarray}
(A^\alpha_i - u^\alpha_i)\prth{\frac{\partial^2 L}{\partial u^\alpha_i \partial u^\beta_j}} = 0, \label{uno}\\
\frac{\partial L}{\partial u^\alpha} - \frac{\partial^2 L}{\partial x^i \partial u^\alpha_i} - A^\beta_i \frac{\partial^2L}{\partial u^\beta \partial u^\alpha_i} - A^\beta_{ji} \frac{\partial^2 L}{\partial u^\beta_j \partial u^\alpha_i} + (A^\beta_j - u^\beta_j) \frac{\partial^2 L}{\partial u^\alpha \partial u^\beta_j} = 0. \label{dos}
\end{eqnarray}
(see \cite{LeonMrnMrr95}). If the lagrangian $L$ is regular, then Eq. \eqref{uno} implies that $A^\alpha_i = u^\alpha_i$ and therefore Eq. \eqref{dos} becomes
\begin{equation} \label{tres}
\frac{\partial L}{\partial u^\alpha} - \frac{\partial^2 L}{\partial x^i \partial u^\alpha_i} - A^\beta_i \frac{\partial^2L}{\partial u^\beta \partial u^\alpha_i} - A^\beta_{ji} \frac{\partial^2 L}{\partial u^\beta_j \partial u^\alpha_i} = 0.
\end{equation}
Now, if $\sigma(x^i)=(x^i, \sigma^\alpha(x), \sigma^\alpha_i(x))$ is an integral section of $\hh$ we would have
\[ u^\alpha_i=\frac{\partial\sigma^\alpha}{\partial x^i} \quad\textrm{and}\quad A^\alpha_{ij}=\frac{\partial\sigma^\alpha_i}{\partial x^j}, \]
which proves that Eq. \eqref{tres} is nothing but the Euler-Lagrange equations for $L$.

We may think Equation \eqref{eq.dynamical.equation} as a generalization of equation \ref{eq.dynamical.equation.rmk} giving  the Euler-Lagrange equations for higher-order field theories in a univocal way, as we will see.
\end{remark}

In a local chart $(x^i,u^\alpha_J,p^{I,i}_\alpha,p)$ of $W_0$, a horizontal projector $\hh$ must have the expression:
\begin{equation} \label{eq.horizontal.projector}
\hh = \prth{\frac{\partial}{\partial x^j}+A^\alpha_{Jj}\frac{\partial}{\partial u^\alpha_J}+B^{Ii}_{\alpha j}\frac{\partial}{\partial p^{I,i}_\alpha}+C_j\frac{\partial}{\partial p}}\tensor\dd x^j,
\end{equation}
where $|I|=0,\dots,k-1$ and $|J|=0,\dots,k$. We then obtain that
\begin{eqnarray*}
i_\hh\Omega_{H_0}-(m\!-\!1)\Omega_{H_0} &\!=\!&\! \prth{B^{Ii}_{\alpha i}\du^\alpha_I-A^\alpha_{Ii}\dd p^{I,i}_\alpha+p^{I,i}_\alpha\du^\alpha_{I+1_i}+u^\alpha_{I+1_i}\dd p^{I,i}_\alpha-\frac{\partial L}{\partial u^\alpha_J}\du^\alpha_J}\wedge\dmx\\
&\!=\!&\! \left[\prth{B^{\ i}_{\alpha i}-\frac{\partial L}{\partial u^\alpha}}\du^\alpha + \sum_{|I'|=1}^{k-1}\prth{B^{I'i}_{\alpha i}-\frac{\partial L}{\partial u^\alpha_{I'}}}\du^\alpha_{I'} + \sum_{|I|=0}^{k-2}p^{I,i}_\alpha\du^\alpha_{I+1_i}\right.\\
&&\! - \sum_{|K|=k}\frac{\partial L}{\partial u^\alpha_K}\du^\alpha_K + \sum_{|I|=k-1}p^{I,i}_\alpha\du^\alpha_{I+1_i}\\
&&\! \left. + \sum_{|I|=0}^{k-1}\prth{u^\alpha_{I+1_i}-A^\alpha_{Ii}}\dd p^{I,i}_\alpha\right]\wedge\dmx.
\end{eqnarray*}
Equating this to zero and using Lemma \ref{th.multi-index.lower.sum} from Appendix \ref{sec.app.multi-index.properties}, we have that
\begin{eqnarray}
     A^\alpha_{Ii} \!&\!=\!&\! u^\alpha_{I+1_i},\ |I|=0,\dots,k-1,\ i=1,\dots,m;\\
B^{\ j}_{\alpha j} \!&\!=\!&\! \frac{\partial L}{\partial u^\alpha};\\
    p^{I,i}_\alpha \!&\!=\!&\! \frac{I(i)+1}{|I|+1}\prth{\frac{\partial L}{\partial u^\alpha_{I+1_i}} - B^{I+1_ij}_{\alpha j} + Q^{Ii}_\alpha}\!,\ |I|=0,\dots,k-2,\ i=1,\dots,m; \label{eq.horizontal.projector.bs.with.q}\\
    p^{I,i}_\alpha \!&\!=\!&\! \frac{I(i)+1}{|I|+1}\prth{\frac{\partial L}{\partial u^\alpha_{I+1_i}} + Q^{Ii}_\alpha}\!,\ |I|=k-1,\ i=1,\dots,m; \label{eq.w1.with.q}
\end{eqnarray}
where the $Q$'s are arbitrary functions such that
\begin{equation} \label{eq.qs}
\sum_{I+1_i=J}\frac{I(i)+1}{|I|+1}Q^{Ii}_\alpha=0, \textrm{ with } |J|=1,\dots,k.
\end{equation}

\begin{remark} \label{rmk.ambiguity}
The ambiguity in the definition of the Legendre transform, and therefore of the Cartan form, becomes apparent in the equations \eqref{eq.horizontal.projector.bs.with.q} and \eqref{eq.w1.with.q}, as noted by Crampin and Saunders (see \cite{SndCrmp90}). There are too many momentum variables to be related univocally with the velocity counterpart. To fix this, a choice of arbitrary functions $Q$ satisfying \eqref{eq.qs} must be done. The choice may be encoded as an additional geometric structure, like a connection.
\end{remark}

Applying \eqref{eq.qs} to \eqref{eq.horizontal.projector.bs.with.q} and \eqref{eq.w1.with.q}, and using the identity \eqref{eq.multi-index.identity}, we finally obtain the equations
\begin{eqnarray}
               A^\alpha_{Ii} &=& u^\alpha_{I+1_i}, \textrm{ with } |I|=0,\dots,k-1,\ i=1,\dots,m; \label{eq.horizontal.projector.as}\\
          B^{\ j}_{\alpha j} &=& \frac{\partial L}{\partial u^\alpha}; \label{eq.horizontal.projector.b1}\\
\sum_{I+1_i=J}p^{I,i}_\alpha &=& \frac{\partial L}{\partial u^\alpha_J} - B^{Jj}_{\alpha j}, \textrm{ with } |J|=1,\dots,k-1; \label{eq.horizontal.projector.bs}\\
\sum_{I+1_i=K}p^{I,i}_\alpha &=& \frac{\partial L}{\partial u^\alpha_K}, \textrm{ with } |K|=k. \label{eq.w1}
\end{eqnarray}

Notice that equation \eqref{eq.w1} is the constraint that defines the space $W_1$; and that \eqref{eq.horizontal.projector.as}, \eqref{eq.horizontal.projector.b1} an d\eqref{eq.horizontal.projector.bs} are conditions on coefficients of the horizontal projectors $\hh$. Note also that, for the time being, the $A$'s with greatest order index and the $C$'s remain undetermined, as well as the most part of the $B$'s.
From the definition of $W_1$, we know that for each point $w \in W_1$ there exists a horizontal projector $\hh_w : T_wW_0 \longrightarrow T_w W_0$ satisfying equation \eqref{eq.dynamical.equation}. However, we cannot ensure that such $\hh_w$, for each $w \in W_1$, will take values in $T_wW_1$. Therefore, we impose  the natural regularizing condition $\hh_w(T_wW_0)\subset T_wW_1$, $\forall w\in W_1$. This latter condition is equivalent to having
\[ \hh\prth{\frac{\partial}{\partial x^j}}\prth{\sum_{I+1_i=K}p^{I,i}_\alpha - \frac{\partial L}{\partial u^\alpha_K}} = 0, \]
which in turn is equivalent (using \eqref{eq.horizontal.projector} and \eqref{eq.horizontal.projector.as}) to
\begin{equation} \label{eq.horizontal.projector.ak}
\sum_{I+1_i=K}B^{Ii}_{\alpha j} = \frac{\partial^2L}{\partial x^j\partial u^\alpha_K} + \sum_{|I|=0}^{k-1}u^\beta_{I+1_j}\frac{\partial^2L}{\partial u^\beta_I\partial u^\alpha_K} + \sum_{|J|=k}A^\beta_{Jj}\frac{\partial^2L}{\partial u^\beta_J\partial u^\alpha_K},
\end{equation}
with $|K|=k$. Thus, if the matrix of second order partial derivatives of $L$ with respect to the ``velocities'' of highest order
\begin{equation}\label{eq.regular.lagrangian}
\prth{\frac{\partial^2L}{\partial u^\beta_J\partial u^\alpha_K}}
\end{equation}
is non-degenerate, then the highest order $A$'s are completely determined in terms of the highest order $B$'s. In the sequel, we will say that the Lagrangian $L: J^1\pi \longrightarrow \RR$ is \emph{regular} if, for any system of adapted coordinates the matrix, \eqref{eq.regular.lagrangian} is non-degenerate.

Up to now, no meaning has been assigned to the coordinate $p$. Consider the submanifold $W_2$ of $W_1$ defined by the restriction $H_0=0$. In other words, $W_2$ is locally characterized by the equation
\[ p = L-p^{I,i}_\alpha u^\alpha_{I+1_i}. \]
As before, we cannot ensure that a solution $\hh$ of the dynamical equation \eqref{eq.dynamical.equation} takes values in $TW_2$. We thus impose to $\hh$ the regularizing condition $\hh_w(T_wW_0)\subset T_wW_2$, $\forall w\in W_2$, or equivalently $\hh(\partial/\partial x^j)(H_0)=0$. Therefore, the coefficients of the linear mapping $\hh$ are governed by the equations \eqref{eq.horizontal.projector.as}, \eqref{eq.horizontal.projector.b1}, \eqref{eq.horizontal.projector.bs},  \eqref{eq.horizontal.projector.ak} and in addition
\begin{equation} \label{eq.horizontal.projector.cs}
C_j = \frac{\partial L}{\partial x^j} + A^\alpha_{Jj}\frac{\partial L}{\partial u^\alpha_J} - A^\alpha_{I+1_ij}p^{I,i}_\alpha - B^{Ii}_{\alpha j}u^\alpha_{I+1_i}.
\end{equation}
Note that, thanks to the Lemma \ref{th.multi-index.lower.sum} and equation \eqref{eq.w1}, the terms with $A$'s with multi-index of length $k$ cancel out, and the $A$'s with lower multi-index are already determined. So, in some sense, the $C$'s depend only on the $B$'s.

\subsection{Description of the solutions} \label{sec.srf.description}
The relations \eqref{eq.horizontal.projector.bs} (with $|J|=k-1$) and \eqref{eq.horizontal.projector.ak} can be seen as a system of linear equations with respect to the $B$'s. When $k=1$, equation \eqref{eq.horizontal.projector.b1} should be considered instead of equation \eqref{eq.horizontal.projector.bs}. In the following, we are going to suppose that $n=1$, since the dimension of the fibres is irrelevant for our purposes and we may ignore it. The number of $B$'s with order $k-1$ (with multi-index length $k-1$) is given by
\[ \binom{m-1+k-1}{m-1}\cdot m^2 \]
and the number of equations with such $B$'s is
\[ \binom{m-1+k}{m-1}\cdot m+\binom{m-1+k-1}{m-1}. \]
An easy computation shows that the system is overdetermined if and only if $k=1$ or $m=1$ (examples \ref{sec.1st.order.lagrangian} and \ref{sec.mechanical.system}), and completely determined when $k=m=2$. In all other cases the system is underdetermined, but it still has maximal rank.

\begin{proposition} \label{th.maximal.rank}
Suppose that $k\geq2$ and $m\geq2$. Then, the system of linear equations with respect to the $B$'s
\begin{eqnarray}
 \sum_{j=1}^m B^{Jj}_{\ j} &=& \frac{\partial L}{\partial u_J} - \sum_{I+1_i=J}p^{I,i}; \label{eq.linear.system.a}\\
\sum_{I+1_i=K}B^{Ii}_{\ j} &=& \frac{\partial^2L}{\partial x^j\partial u_K} + \sum_{|I|=0}^{k-1}u_{I+1_j}\frac{\partial^2L}{\partial u_I\partial u_K} + \sum_{|J|=k}A_{Jj}\frac{\partial^2L}{\partial u_J\partial u_K}; \label{eq.linear.system.b}
\end{eqnarray}
where $|J|=k-1$, $j=1,\dots,m$ and $|K|=k$, has maximal rank.
\end{proposition}

\begin{proof}
In a first step, we are going to describe how to write the matrix of coefficients. Then, we will select the proper columns of this matrix to obtain a new square matrix of maximal size. We finally shall prove that this matrix has maximal rank.

The matrix of coefficients will be a rectangular matrix formed by 1's and 0's. The columns will be indexed by the indexes of the $B$'s, and the rows by the indexes of the first partial derivatives that appear in the equations \eqref{eq.linear.system.a} and \eqref{eq.linear.system.b}. As  $B^{Ii}_{\ j}$ has three indexes, the columns of the matrix of coefficients will organized in a superior level by the index $i$, in a middle level by the index $j$ and in an inferior level by the multi-index $I$. The rows will be organized at the top by the index $J$ for the first equation, \eqref{eq.linear.system.a}, and at the bottom by the index $j$ and then by the multi-index $K$ for the second equation, \eqref{eq.linear.system.b}.

As the matrix of coefficients has more columns than rows, we shall build a second matrix that has as many columns and rows as the matrix of coefficients has rows. To do that, we select a column of the matrix of coefficients for each row index using the following \emph{algorithm} (for the sake of simplicity):
\begin{verbatim}
01   ForEach (j,K)
02      Define G={(I,i):I+1_i=K}
03      If Cardinal(G)=1
04         Select the column (i,j,I)
05      ElseIf
06         Select a column (i,j,I) such that (I,i) is in G and i\neq j
07      EndIf
08   EndFor
09   ForEach J
10      If J(1)=k-1
11         Select the column (m,m,J)
12      ElseIf
13         Select the column (1,1,J)
14      EndIf
15   EndFor
\end{verbatim}

Now, this matrix being defined and since it is full of 0's and has only few 1's, we are going to develop its determinant by rows and columns. Notice that the columns selected at line 6 have only one 1 each, thus we can cross out the rows an columns related to these 1's. Now the rows at the bottom part of the remaining matrix (related to the second equations) have only one 1 each, thus we can also cross out the rows an columns related to these 1's. Now, the remaining matrix has the property of having only one 1 per column and row (there must be at least one 1 per row and column, and no two 1's may be at the same row or column), thus its determinant is not zero and the matrix of coefficients has maximal rank.
\end{proof}

Another way to interpret the tangency condition \eqref{eq.horizontal.projector.ak} is the following one: Let us suppose we are dealing with a first order Lagrangian (example \ref{sec.1st.order.lagrangian}, equation \eqref{eq.1st.order.tangency.bs}). One could apply the theory of connections to the Lagrangian setting and the Hamiltonian one as separate frameworks. We know that they must be related by means of the Legendre transform and so are the horizontal projectors induced by these connections. Thus, equation \eqref{eq.1st.order.tangency.bs} is nothing else than the relation between the coefficients of these horizontal projectors.

\subsection{The reduced mixed space $W_2$} \label{sec.srf.reduced.space}
In section \secref{sec.srf.description} we reduced the space $W_1$ to $W_2$ by considering the constraint $H_0=0$, which is a way of interpreting the coordinate $p$ as the Hamiltonian function. But $W_2$ is not a mere instrument to get rid off the coordinate $p$ or the coefficients $C_j$. As the premultisymplectic form $\Omega_{H_0}$, it encodes the dynamics of the system and, when $L$ is regular, it is a multisymplectic space. Indeed, when $k=1$, $W_2$ is diffeomorphic to $J^1\pi$ (\cf\ de Le\'on \etal\ \cite{LeonMrrMrt03}), which is not true for higher order cases.

\begin{proposition} \label{th.multisymplectic.iff.regular}
Let $W_2=\set{w\in W_1\ :\ H_0(w)=0}$ and define the $(m+1)$-form $\Omega_2$ as the pullback of $\Omega_{H_0}$ to $W_2$ by the natural inclusion $i:W_2\hookrightarrow W_0$, that is $\Omega_2=i^*(\Omega_{H_0})$. Suppose that $\dim M>1$, then, the $(m+1)$-form $\Omega_2$ is multisymplectic if and only if $L$ is regular.
\end{proposition}

\begin{proof}
First of all, let us make some considerations. By definition, $\Omega_2$ is multisymplectic whenever $\Omega_2$ has trivial kernel, that is,
\[\hbox{\rm if }  v\in TW_2, i_v\Omega_2=0\ \Longleftrightarrow\ v=0\ . \]
This is equivalent to say that
\[ \hbox{\rm if } v\in i_*(TW_2), i_v\Omega_{H_0}|_{i_*(TW_2)}=0\ \Longleftrightarrow\ v=0\ . \]
Let $v\in TW_0$ be a tangent vector whose coefficients in an adapted basis are given by
\[ v = \lambda^i\frac{\partial}{\partial x^i} + A^\alpha_J\frac{\partial}{\partial u^\alpha_J} + B^{Ii}_\alpha\frac{\partial}{\partial p^{Ii}_\alpha} + C\frac{\partial}{\partial p}. \]
Using the expression \eqref{eq.hamiltonian.form.coordinates}, we may compute the contraction of $\Omega_{H_0}$ by $v$,
\begin{equation} \label{eq.ivomega0}
\begin{array}{rcl}
i_v\Omega_{H_0} &=& - B^{Ii}_\alpha\du^\alpha_I\wedge\dmxi + A^\alpha_I\dd p^{Ii}_\alpha\wedge\dmxi- \lambda^j\dd p^{Ii}_\alpha\wedge\du^\alpha_I\wedge\dd^{m-2}x_{ij}\\
            & & + \prth{A^\alpha_{I+1_i}p^{Ii}_\alpha + B^{Ii}_\alpha u^\alpha_{I+1_i} - A^\alpha_J\frac{\partial L}{\partial u^\alpha_J}}\dmx\\
            & & - \lambda^j\prth{p^{Ii}_\alpha\du^\alpha_{I+1_i} + u^\alpha_{I+1_i}\dd p^{Ii}_\alpha - \frac{\partial L}{\partial u^\alpha_J}\du^\alpha_J}\wedge\dd^{m-1}x_j.
\end{array}
\end{equation}
On the other hand, if we now suppose that $v$ is tangent to $W_2$ in $W_0$, that is $v\in i_*(TW_2)$, we then have that
\begin{equation}\label{tangency}
 \dd\prth{\sum_{I+1_i=K}p^{Ii}_\alpha-\frac{\partial L}{\partial u^\alpha_K}}(v)=0 \quad\textrm{and}\quad \dd H_0(v)=0, \end{equation}
which leads us to the following relations for the coefficients of $v$,
\begin{eqnarray}
&\displaystyle{ \sum_{I+1_i=K}B^{Ii}_\alpha = \lambda^i\frac{\partial^2L}{\partial x^i\partial u^\alpha_K} + A^\beta_J\frac{\partial^2L}{\partial u^\beta_J\partial u^\alpha_K}} \textrm{ and} &\label{eq.dw1v}\\
& A^\alpha_{I+1_i}p^{Ii}_\alpha + B^{Ii}_\alpha u^\alpha_{I+1_i} + C - \lambda^i\frac{\partial L}{\partial x^i} - A^\alpha_J\frac{\partial L}{\partial u^\alpha_J} = 0. &\label{eq.dh0v}
\end{eqnarray}
It is important to note that thanks to Lemma \ref{th.multi-index.fubini} and the equation \eqref{eq.w1} which defines $W_1$ (and hence $W_2$), the terms in \eqref{eq.ivomega0} and \eqref{eq.dh0v} involving $A$'s with multi-index of length $k$ cancel each other out.

These considerations being made, suppose that $\Omega_2$ is multisymplectic and, by \emph{reductio ad absurdum}, suppose in addition that $L$ is not regular, which means that the matrix
\[ \prth{\frac{\partial^2L}{\partial u^\beta_{K'}\partial u^\alpha_K}} \]
has non-trivial kernel. Let $v\in TW_0$ be a tangent vector such that all its coefficients are null except the $A$'s of highest order which are in such a way they are mapped to zero by the ``hessian'' of $L$. Such a vector $v$ fulfills the restrictions \eqref{eq.dw1v} and \eqref{eq.dh0v}, thus it must be tangent to $W_2$ in $W_0$, $v\in i_*(TW_2)$. But, as $i_v\Omega_{H_0}$ has no $A$'s of highest order, it must be zero, $i_v\Omega_{H_0}=0$, which is a contradiction.

Conversely, let us suppose that $L$ is regular, then equation \eqref{eq.w1} defines implicitly the coordinates $u^\alpha_K$ as functions of the other coordinates. That is, locally there exist functions $f^\alpha_K(x^i,u^\alpha_I,p^{I,i}_\alpha)$ such that $u^\alpha_K=f^\alpha_K$ on $i(W_2)$. Furthermore, given a system of adapted coordinates $(x^i,u^\alpha_I,u^\alpha_K,p^{I,i}_\alpha,p)$ on $W_0$, $(x^i,u^\alpha_I,p^{I,i}_\alpha)$ defines a coordinate system on $W_0$ and the inclusion is given by:
\[ (x^i,u^\alpha_I,p^{I,i}_\alpha)\in W_2 \hookrightarrow (x^i,u^\alpha_I,f^\alpha_K,p^{I,i}_\alpha,L-\sum_{|I|=0}^{k-2}p^{I,i}_\alpha u^\alpha_{I+1_i}-\sum_{|I|=k-1}p^{I,i}_\alpha f^\alpha_{I+1_i})\in W_0. \]
From equation \eqref{eq.hamiltonian.form.coordinates}, we can compute an explicit expression of the $(m+1)$-form $\Omega_2$ in this coordinate system,
\begin{eqnarray*}
\Omega_2 &=& -\sum_{|I|=0}^{k-1}\dd p^{I,i}_\alpha\wedge\du^\alpha_I\wedge\dmxi\\
         & & + \lie{\sum_{|I|=0}^{k-2}\prth{p^{I,i}_\alpha\du^\alpha_{I+1_i}+u^\alpha_{I+1_i}\dd p^{I,i}_\alpha} - \sum_{|I|=0}^{k-1}\frac{\partial L}{\partial u^\alpha_I}\du^\alpha_I}\wedge\dmx\\
         & & + \lie{\sum_{|I|=k-1}\prth{p^{I,i}_\alpha\dd f^\alpha_{I+1_i}+f^\alpha_{I+1_i}\dd p^{I,i}_\alpha} - \sum_{|K|=k}\sum_{I+1_i=K}p^{I,i}_\alpha\dd f^\alpha_K}\wedge\dmx,
\end{eqnarray*}
where we have used equation \eqref{eq.w1} in the last term. Note that, by Lemma \ref{th.multi-index.fubini}, the first and last terms of the last bracket cancel each other out. Now,
\begin{eqnarray*}
i_{\partial/\partial x^j}\Omega_2 &=& \dd p^{I,i}_\alpha\wedge\du^\alpha_I\wedge\dd^{m-2}x_{ij} - [\dots]\wedge\dd^{m-1}x_j\\
i_{\partial/\partial u^\alpha_I}\Omega_2 &=& \dd p^{I,i}_\alpha\wedge\dmxi + \prth{\sum_{J+1_j=I}p^{Jj}_\alpha-\frac{\partial L}{\partial u^\alpha_I}}\dmx\\
i_{\partial/\partial p^{I,i}_\alpha}\Omega_2 &=& \du^\alpha_I\wedge\dmxi + u^\alpha_{I+1_i}\dmx,
\end{eqnarray*}
where $0\leq|I|\leq k-1$. We deduce from here that the kernel of $\Omega_2$ is trivial, $\ker\Omega_2=\set{0}$, and $\Omega_2$ is multisymplectic.
\end{proof}

\begin{note}
In the particular case  when $\dim M=1$, the Lagrangian function $L: J^k\pi\longrightarrow \R$ is regular if and only if the pair $(\Omega_{2}, \tau_{W_2,M}^*dt)$ is a \emph{cosymplectic structure} on $W_2$. We recall that a cosymplectic structure on a manifold $N$ of odd dimension $2\bar{n}+1$ is a pair which consists of a closed $2$-form $\Omega$ and a closed $1$-form $\eta$ such that $\eta \wedge \Omega^{\bar{n}}$ is a volume form.
\end{note}

We remark that, if the Lagrangian $L$ is regular or (from Proposition \ref{th.maximal.rank}) if $k,m>1$, then there locally exist solutions $\hh$ of the dynamical equations \eqref{eq.dynamical.equation} on $W_2$ that give rise to connections $\Gamma$ in the fibration $\pi_{W_0 M} : W_0 \longrightarrow M$ along the submanifold $W_2$ (see Appendix \ref{sec.app.connections}). In such a case, a global solution is obtained using partitions of the unity, and we obtain by restriction a connection $\bar{\Gamma}$, with horizontal projector $\bar{\hh}$, in the fibre bundle $\pi_{W_2 M} : W_2 \longrightarrow M$, which is a solution of equation \eqref{eq.dynamical.equation} when it is restricted to $W_2$ (in fact, we have a family of such solutions).

In some cases, but only when $\dim M=1$ or $k=1$, it would be necessary to consider a subset $W_3$ defined in order to satisfy the tangency conditions \eqref{eq.horizontal.projector.ak} and \eqref{eq.horizontal.projector.cs}:
\begin{eqnarray*}
W_3 &=& \{ w\in W_2\ /\ \exists\hh_w:T_wW_0\To T_wW_2\textrm{ linear such that }\hh_w^2=\hh_w,\\
    & & \ker\hh_w=(V\pi_{W_0,M})_w,\ i_{\hh_w}\Omega_{H_0}(w)=(m-1)\Omega_{H_0}(w) \}.
\end{eqnarray*}
We will assume that ${W}_3$ is a submanifold of $W_2$. If $\hh_w(T_wW_0)$ is not contained in $T_wW_3$, we go to the third step, and so on. At the end, and if the system has solutions, we will find a final constraint submanifold $W_f$, fibered over $M$ (or over some open subset of $M$) and a connection ${\Gamma}_f$ in this fibration such that ${\Gamma}_f$ is a solution of equation \eqref{eq.dynamical.equation} restricted to ${W}_f$.

In any case, one obtains the Euler-Lagrange equations. In the following result, $W_f$ denotes the final constraint manifold, which is $W_2$ when $k,m>1$, and $\hh$ the horizontal projector of a connection in $\pi_{W_2,M}:W_f\To M$ along $W_f$, which is solution of the dynamical equation.

\begin{proposition} \label{th.euler-lagrange.equations}
Let $\bar{\sigma}$ be a section of $\pi_{W_f,M}:W_f\To M$ and denote $\sigma=i\circ\bar\sigma$, where $i: W_f\hookrightarrow W_0$ is the canonical inclusion. If $\bar\sigma$ is an integral section of $\hh$, then $\bar{\sigma}$ is \emph{holonomic}, in the sense that
\begin{equation} \label{eq.holonomic.section}
\pr_1\circ\sigma = j^k(\pi_{W_f,E}\circ\bar{\sigma}),
\end{equation}
and satisfies the higher-order Euler-Lagrange equations:
\begin{equation} \label{eq.euler-lagrange.equations}
j^{2k}(\pi_{W_f,E}\circ\bar\sigma)^*\prth{\sum_{|J|=0}^k(-1)^{|J|}\frac{\dd^{|J|}}{\dx^J}\frac{\partial L}{\partial u^\alpha_J}} = 0.
\end{equation}
\end{proposition}

\begin{proof}
If $\sigma=(x^i,\sigma^\alpha_J,\sigma^{I,i}_\alpha,\tilde\sigma)$ is an integral section of $\hh$, then
\[ \frac{\partial\sigma^\alpha_J}{\partial x^j}=A^\alpha_{Jj},\ \frac{\partial\sigma^{Ii}_\beta}{\partial x^j}=B^{Ii}_{\beta j} \textrm{ and } \frac{\partial\tilde\sigma}{\partial x^j}=C_j, \]
where the $A$'s, $B$'s and $C$'s are the coefficients given in \eqref{eq.horizontal.projector}. From equation \eqref{eq.horizontal.projector.as}, we have that $\sigma$ is holonomic, in the sense that $\sigma^\alpha_{I+1_i}=\partial\sigma^\alpha_I/\partial x^i$. On the other hand, using the equations \eqref{eq.horizontal.projector.b1}, \eqref{eq.horizontal.projector.bs} and \eqref{eq.w1}, we obtain the relations (where $\phi=pr_1\circ\sigma$):
\begin{eqnarray}
                                0 &=& \frac{\partial L}{\partial u^\alpha}\circ\phi - \frac{\partial\sigma^{\ j}_\alpha}{\partial x^j}; \label{eq.integral.sections.0}\\
\sum_{I+1_i=J}\sigma^{I,i}_\alpha &=& \frac{\partial L}{\partial u^\alpha_J}\circ\phi - \frac{\partial\sigma^{Jj}_\alpha}{\partial x^j}, \textrm{ with } |J|=1,\dots,k-1; \label{eq.integral.sections.i}\\
\sum_{I+1_i=K}\sigma^{I,i}_\alpha &=& \frac{\partial L}{\partial u^\alpha_K}\circ\phi, \textrm{ with } |K|=k. \label{eq.integral.sections.k}
\end{eqnarray}
From the equations \eqref{eq.integral.sections.0} and \eqref{eq.integral.sections.i} for $|J|=1$ we get
\begin{eqnarray*}
0 &=& \frac{\partial L}{\partial u^\alpha}\circ\phi - \frac{\partial\sigma^{\ j}_\alpha}{\partial x^j}\\
  &=& (j^0\phi)^*\frac{\partial L}{\partial u^\alpha} - \sum_{|I|=1}(j^1\phi)^*\prth{\frac{\dd^{|I|}}{\dx^I}\frac{\partial L}{\partial u^\alpha_I}} + \sum_{|I|=1}\sum_i\frac{\partial^{|I|}}{\partial x^I}\frac{\partial\sigma^{Ii}_\alpha}{\partial x^i}.
\end{eqnarray*}
Applying now Lemma \ref{th.multi-index.fubini} on the last term and repeating this process until $|I|=k-1$ we reach
\begin{eqnarray*}
0 &=& (j^0\phi)^*\frac{\partial L}{\partial u^\alpha} - \sum_{|I|=1}(j^1\phi)^*\prth{\frac{\dd^{|I|}}{\dx^I}\frac{\partial L}{\partial u^\alpha_I}} + \sum_{|J|=2}\frac{\partial^{|J|}}{\partial x^J}\sum_{I+1_i=J}\sigma^{Ii}_\alpha\\
  &=& (j^0\phi)^*\frac{\partial L}{\partial u^\alpha} - \sum_{|I|=1}(j^1\phi)^*\prth{\frac{\dd^{|I|}}{\dx^I}\frac{\partial L}{\partial u^\alpha_I}} + \sum_{|I|=2}(j^2\phi)^*\prth{\frac{\dd^{|I|}}{\dx^I}\frac{\partial L}{\partial u^\alpha_I}} - \sum_{|I|=2}\sum_i\frac{\partial^{|I|}}{\partial x^I}\frac{\partial\sigma^{Ii}_\alpha}{\partial x^i}\\
  &=& \sum_{|I|=0}^{k-1}(-1)^{|I|}(j^{|I|}\phi)^*\prth{\frac{\dd^{|I|}}{\dx^I}\frac{\partial L}{\partial u^\alpha_I}} - (-1)^{k-1}\sum_{|I|=k-1}\sum_i\frac{\partial^{|I|}}{\partial x^I}\frac{\partial\sigma^{Ii}_\alpha}{\partial x^i}\\
  &=& \sum_{|I|=0}^{k-1}(-1)^{|I|}(j^{|I|}\phi)^*\prth{\frac{\dd^{|I|}}{\dx^I}\frac{\partial L}{\partial u^\alpha_I}} - (-1)^{k-1}\sum_{|K|=k}\frac{\partial^{|K|}}{\partial x^K}\sum_{I+1_i=K}\sigma^{Ii}_\alpha,
\end{eqnarray*}
where by abuse of notation $j^l\phi=j^{k+l}(\pi_{W_f,E}\circ\bar{\sigma})$. Finally, it only rest to use equation \eqref{eq.integral.sections.k} to prove the desired result.
\end{proof}

\section{Examples}
First, we are going to study the particular cases when $k=1$ and $m=1$, which correspond to the First Order Classical Field Theory and to the Higher Order Mechanical Systems, respectively. Theoretic results for these cases are very well known \cite{BntMrt05,LeonMrrMrt03,MrnMnzRmn03,GrcPnsRmn91} and we are only going to recover these results from our general setting.

\begin{example}[First order Lagrangians ($k=1$)] \label{sec.1st.order.lagrangian}
Let us suppose that $k=1$, which corresponds to the case of first order Lagrangians. In that case the velocity-momentum space is $W_0=J^1\pi\otimes_E\Lambda^m_2E$, with adapted coordinates $(x^i,u^\alpha,u^\alpha_i,p,p^i_\alpha)$. The premultisymplectic $(m+1)$-form would be
\begin{equation}
\Omega_{H_0} = -\dd p^i_\alpha\wedge\du^\alpha\wedge\dmxi+\prth{p^i_\alpha\du^\alpha_i+u^\alpha_i\dd p^i_\alpha-\frac{\partial L}{\partial u^\alpha}\du^\alpha-\frac{\partial L}{\partial u^\alpha_i}\du^\alpha_i}\wedge\dmx,
\end{equation}
and horizontal projectors on $TW_0$ would have locally the form:
\begin{equation}
\hh = \prth{\frac{\partial}{\partial x^j}+A^\alpha_j\frac{\partial}{\partial u^\alpha}+A^\alpha_{ij}\frac{\partial}{\partial u^\alpha_i}+B^{\ i}_{\alpha j}\frac{\partial}{\partial p^i_\alpha}+C_j\frac{\partial}{\partial p}}\tensor\dd x^j.
\end{equation}
Solutions of the dynamical equation would satisfy the relations
\begin{eqnarray}
\sum_{j=1}^mB^{\ j}_{\alpha j} &=& \frac{\partial L}{\partial u^\alpha}; \label{eq.1st.order.skrs.bs} \\
                    p^i_\alpha &=& \frac{\partial L}{\partial u^\alpha_i}, \textrm{ for } i=1,\dots,m; \\
                    A^\alpha_i &=& u^\alpha_i, \textrm{ for } i=1,\dots,m;
\end{eqnarray}
from which we deduce the Euler-Lagrange equations
\begin{equation}
j^2(\pi_{W_2,M}\circ\sigma)^*\prth{\frac{\partial L}{\partial u^\alpha}-\sum_{i=1}^m\frac{\dd}{\dx^i}\frac{\partial L}{\partial u^\alpha_i}} = 0,
\end{equation}
where $W_2$ is defined by
\begin{equation}
W_2 = \set{ (x^i,u^\alpha,u^\alpha_i,p,p^i_\alpha)\in W_1\ :\ p^i_\alpha=\frac{\partial L}{\partial u^\alpha_i}, \ p=L-\sum_{i=1}^mp^iu_i }.
\end{equation}
We then obtain the tangency conditions:
\begin{eqnarray}
B^{\ i}_{\alpha j} &=& \frac{\partial^2L}{\partial x^j\partial u^\alpha_i} + u^\beta_j\frac{\partial^2L}{\partial u^\beta\partial u^\alpha_i} + \sum_{l=1}^mA^\beta_{lj}\frac{\partial^2L}{\partial u^\beta_l\partial u^\alpha_i}, \label{eq.1st.order.tangency.bs}\\
C_j &=& \frac{\partial L}{\partial x^j} + u^\alpha_j\frac{\partial L}{\partial u^\alpha} - B^{\ i}_{\alpha j}u^\alpha_i.
\end{eqnarray}
Note that \eqref{eq.1st.order.tangency.bs} is the relation that would appear between the coefficients of a Lagrangian and a Hamiltonian setting through the Legendre transform. For simplicity, suppose that $n=1$ and ignore the $\alpha$'s and $\beta$'s that appear above. Consider the linear system of equations with respect to the $B$'s formed by equations \eqref{eq.1st.order.skrs.bs} and \eqref{eq.1st.order.tangency.bs}. This system is overdetermined since it has $m^2+1$ equations and only $m^2$ variables ($B^i_j$).
\end{example}

\begin{example}[Higher order mechanical systems ($m=1$)] \label{sec.mechanical.system}
Let us suppose that $m=1$, which corresponds to the case of mechanical systems. In that case the velocity-momentum space is $W_0=J^k\pi\times_{J^{k-1}\pi}\Lambda^m_2(J^{k-1}\pi)$. Since here a multi-index $J$ is of the form $(l)$ with $1\leq l\leq k$, we change the usual notation for coordinates to
\[ u^\alpha_J\longrightarrow u^\alpha_{|J|} \quad\textrm{and}\quad p^{I,1}_\alpha\longrightarrow p^{|I|+1}_\alpha, \]
and we adapt the remaining objects to this notation. So adapted coordinates on $W_0$ are of the form $(x,u^\alpha, u^\alpha_l,p,p^l_\alpha)$, where $l=1,\dots,k$. The premultisymplectic $(m+1)$-form would be
\begin{equation}
\Omega_{H_0} = -\sum_{l=0}^{k-1}\dd p^{l+1}_\alpha\wedge\du^\alpha_l+\sum_{l=1}^k\prth{p^l_\alpha\du^\alpha_l+u^\alpha_l\dd p^l_\alpha}\wedge\dx-\sum_{l=0}^k\frac{\partial L}{\partial u^\alpha_l}\du^\alpha_l\wedge\dx,
\end{equation}
and horizontal projectors on $TW_0$ would have locally the form:
\begin{equation}
\hh = \prth{\frac{\partial}{\partial x}+\sum_{l=0}^kA^\alpha_l\frac{\partial}{\partial u^\alpha_l}+\sum_{l=1}^kB^l_\alpha\frac{\partial}{\partial p^l_\alpha}+C\frac{\partial}{\partial p}}\tensor\dx.
\end{equation}
Solutions of the dynamical equation would satisfy the relations
\begin{eqnarray}
B^1_\alpha &=& \frac{\partial L}{\partial u^\alpha}; \\
p^l_\alpha &=& \frac{\partial L}{\partial u^\alpha_l} - B^{l+1}_\alpha, \textrm{ for } l=1,\dots,k-1;\label{eq.mechanical.systems.skrs.bs} \\
p^k_\alpha &=& \frac{\partial L}{\partial u^\alpha_k};\label{eq.mechanical.systems.skrs.w1} \\
A^\alpha_l &=& u^\alpha_{l+1}, \textrm{ for } l=0,\dots,k-1.
\end{eqnarray}
which we use to get the Euler-Lagrange equations
\begin{equation}
j^{2k}(\pi_{W_2,M}\circ\sigma)^*\prth{\sum_{l=0}^k(-1)^l\frac{\dd^l}{\dx^l}\frac{\partial L}{\partial u^\alpha_l}} = 0,
\end{equation}
where $W_2$ is defined by
\begin{equation}
W_2 = \set{ (x^i,u^\alpha,u^\alpha_l,p,p^l_\alpha)\in W_1\ :\ p^k_\alpha=\frac{\partial L}{\partial u^\alpha_k}, \ p=L-\sum_{l=1}^kp^l_\alpha u^\alpha_l }.
\end{equation}
We then obtain the tangency conditions:
\begin{eqnarray}
B^k_\alpha &=& \frac{\partial^2L}{\partial x\partial u^\alpha_k} + \sum_{l=0}^{k-1}u^\beta_{l+1}\frac{\partial^2L}{\partial u^\beta_l\partial u^\alpha_k} + A^\beta_{k'}\frac{\partial^2L}{\partial u^\beta_{k'}\partial u^\alpha_k} = 0; \label{eq.mechanical.systems.tangencyty.bs}\\
C &=& \frac{\partial L}{\partial x} + \sum_{l=0}^{k-1}u^\alpha_{l+1}\frac{\partial L}{\partial u^\alpha_l} + A^\alpha_k\frac{\partial L}{\partial u^\alpha_k} - \sum_{l=1}^k\prth{A^\alpha_lp^l_\alpha + B^l_{\alpha j}u^\alpha_l}. \label{eq.mechanical.systems.tangencyty.cs}
\end{eqnarray}
Note that, thanks to equation \eqref{eq.mechanical.systems.skrs.w1}, the terms in \eqref{eq.mechanical.systems.tangencyty.cs} with coefficient $A_k$ cancel out. Now, for simplicity, suppose that $n=1$ and ignore the $\alpha$'s and $\beta$'s that appear above. Consider the linear system of equations with respect to the $B$'s formed by equations \eqref{eq.mechanical.systems.skrs.bs} (with $l=k-1$) and \eqref{eq.mechanical.systems.tangencyty.bs}. This system is overdetermined since it has 2 equations and only one variable ($B^k$).
\end{example}

\begin{example}[The loaded and clamped plate]
Let us set $M=\RR^2$ and $E=\RR^2\times\RR=\RR^3$, and consider the Lagrangian
\[ L(x,y,u,u_x,u_y,u_{xx},u_{xy},u_{yy}) = \frac12(u_{xx}^2+2u_{xy}^2+u_{yy}^2-2qu), \]
where $q=q(x,y)$ is the normal load on the plate. Given a regular region $R$ of the plane, we look for the extremizers of the functional $I(u)=\int_RL$ such that $u=\partial u/\partial n=0$ on the border $\partial R$, where $n$ is the normal exterior vector. The Euler-Lagange equation associated to the problem is
\begin{equation}
u_{xxxx}+2u_{xxyy}+u_{yyyy} = q.
\end{equation}
Written in the multi-index notation, the Lagrangian has the form
\[ L(j^2\phi) = \frac12(u_{(2,0)}^2+2u_{(1,1)}^2+u_{(0,2)}^2-2qu) \]
and the Euler-Lagrange equation reads
\[u_{(4,0)}+2u_{(2,2)}+u_{(0,4)} = q. \]
The velocity-momentum space is $W_0=J^2\pi\times_{J^1\pi}\Lambda^2_2(J^1\pi)$, with adapted coordinates $(x,y,u_x,$ $u_y,u_{xx},u_{xy},u_{yy},p,p^x,p^y,p^{yy},p^{xy},p^{yx},p^{yy})$. It is straightforward to write down the premultisymplectic 3-form and a general horizontal projector on $TW_0$, so we are not going to do it here. Even so, the coefficients of solutions of the dynamical equation would satisfy the relations
\begin{equation} \label{eq.plate.horizontal.projector}
B^{\ ,x}_x+B^{\ ,y}_y=-2q \qquad
\begin{array}{rcl}
-p^x &=& B^{x,x}_x + B^{x,y}_y\\
-p^y &=& B^{y,x}_x + B^{y,y}_y
\end{array}
\qquad
\begin{array}{rcl}
       p^{xx} &=& u_{xx}\\
p^{xy}+p^{yx} &=& 2u_{xy}\\
       p^{yy} &=& u_{yy}
\end{array}
\end{equation}
where the latter ones are the equations that define $W_1$. The tangency condition on $W_1$ gives us the relations
\begin{equation} \label{eq.plate.tangency}
\begin{array}{rcl}
          B^{x,x}_x &=&  A_{xx,x}\\
B^{x,y}_x+B^{y,x}_x &=& 2A_{xy,x}\\
          B^{y,y}_x &=&  A_{yy,x}
\end{array}
\qquad
\begin{array}{rcl}
          B^{x,x}_y &=&  A_{xx,y}\\
B^{x,y}_y+B^{y,x}_y &=& 2A_{xy,y}\\
          B^{y,y}_y &=&  A_{yy,y}
\end{array}
\end{equation}
from where we can see that the Lagrangian is ``regular'', since
\begin{equation}
\prth{\frac{\partial^2L}{\partial u_K\partial u_{K'}}}_{|K|=|K'|=2} = \prth{\begin{array}{ccc}1&0&0\\0&2&0\\0&0&1\end{array}}.
\end{equation}
Finally, we remark that the middle equations of \eqref{eq.plate.horizontal.projector} and \eqref{eq.plate.tangency} form a $8\times8$ linear system of equations on the $B$'s, which is completely determined.
\end{example}

\begin{example}[The Camassa-Holm equation]
In 1993, Camassa and Holm introduced the following completely integrable bi-Hamiltonian equation (see \cite{CmssHlm93}):
\begin{equation} \label{eq.camassa-holm.eulerian}
v_t-v_{yyt}=-3vv_y+2v_yv_{yy}+vv_{yyy},
\end{equation}
which is used to model the breaking waves in shallow waters as the Korteweg–de Vries equation. But, as the former is of higher order, we are going to use it as example.

The CH equation \eqref{eq.camassa-holm.eulerian} is expressed in terms of the Eulerian or spatial velocity field $u(y,t)$, and it is the Euler-Poincar\'e equation of the reduced Lagrangian
\begin{equation} \label{eq.camassa-holm.lagrangian.reduced}
l(v)=\frac12\int\prth{v^2+v_y^2}\dy.
\end{equation}
To give a multisymplectic approach to the problem, as Kouranbaeva and Shkoller did (see \cite{KrbShk00}), we must express the CH equation \eqref{eq.camassa-holm.eulerian} in Lagrangian terms. Thus, we shall use the Lagrangian variable $u(x,t)$ that arises as the solution of
\begin{equation} \label{eq.lagrangian.variable}
\frac{\partial u(x,t)}{\partial t} = v(u(x,t),t).
\end{equation}
The independent variables $(x,t)$ are coordinates for the base space $M=S^1\times\RR$, and the dependent variable $u(x,t)$ is a fiber coordinate for the total space $E=S^1\times\RR\times\RR=S^1\times\RR^2$. The Lagrangian action is now written as
\begin{equation} \label{eq.camassa-holm.lagrangian.action}
L(x,t,u,u_x,u_t,u_{xx},u_{xt},u_{tt}) = \frac12(u_xu_t^2+u_x^{-1}u_{xt}^2)
\end{equation}
The coefficients of a horizontal projector which is solution of the dynamical equation must satisfy
\begin{equation} \label{eq.ch.projector}
\begin{array}{rcl}
B^{\ ,x}_x+B^{\ ,t}_t&=&0\\
                 p^x &=& 1/2(u_t^2-(u_{xt}/u_x)^2) - (B^{x,x}_x + B^{x,t}_t)\\
                 p^t &=& u_xu_t - (B^{t,x}_x + B^{t,t}_t)\\
              p^{xx} &=& 0\\
       p^{xt}+p^{tx} &=& u_{xt}/u_x\\
              p^{tt} &=& 0
\end{array}
\end{equation}
where the last three are the equations that define $W_1$. The tangency condition on $W_1$ gives us the relations
\begin{equation} \label{eq.ch.tangency}
\begin{array}{rcl}
          B^{x,x}_x &=&  0\\
B^{x,t}_x+B^{t,x}_x &=& -u_x^{-1}u_{xx}u_{xt}+A_{xt,x}u_x^{-1}\\
          B^{t,t}_x &=&  0\\
          B^{x,x}_t &=&  0\\
B^{x,t}_t+B^{t,x}_t &=& -(u_{xt}/u_x)^2+A_{xt,t}u_x^{-1}\\
          B^{t,t}_t &=&  0
\end{array}
\end{equation}
from where we can see that the Lagrangian is clearly ``singular'', since
\begin{equation}
\prth{\frac{\partial^2L}{\partial u_K\partial u_{K'}}}_{|K|=|K'|=2} = \prth{\begin{array}{ccc}0&0&0\\0&u_x^{-1}&0\\0&0&0\end{array}}
\end{equation}
Again, we may form a completely determined system of linear equations on the $B$'s with the corresponding relations of \eqref{eq.plate.horizontal.projector} and the equations \eqref{eq.ch.tangency}.
\end{example}

\begin{example}[First order Lagrangian as second order]
For the sake of simplicity, let suppose that $n=1$. Given a first order Lagrangian $L:J^1\pi\To\RR$, extend it to a second order Lagrangian, $\bar L=L\circ\pi_{2,1}$. Consider the first and second order velocity-momenta mixed spaces $W^1_0=J^1\pi\times_E\Lambda^m_2E$ and $W^2_0=J^2\pi\times_{J^2\pi}\Lambda^m_2(J^2\pi)$, with adapted coordinates $(x^i,u,u_i,p,p^i)$ and $(x^i,u,u_i,u_K,p,p^i,p^{ij})$ (with $|K|=2$), respectively. Let $\pi^{2,1}_0:W^2_0\To W^1_0$ be the natural projection (Diagram \ref{fig.1st-2nd.order}).

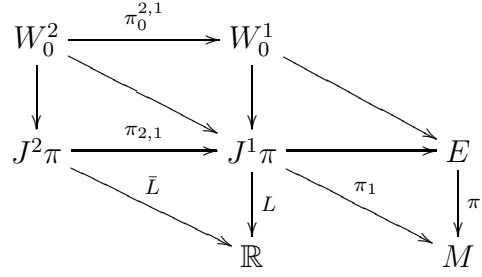
\begin{figure}[h]
\[\xymatrix{
  W^2_0  \ar[rr]^{\pi^{2,1}_0} \ar[drr] \ar[d]   && W^1_0  \ar[drr] \ar[d] &&  \\
  J^2\pi \ar[rr]^{\pi_{2,1}}   \ar[drr]^{\bar L} && J^1\pi \ar[rr] \ar[drr]^{\pi_1} \ar[d]^L && E \ar[d]^\pi\\
                                                 && \RR                                      && M
  }\]
\caption{The 1st and 2nd order Lagrangian settings} \label{fig.1st-2nd.order}
\end{figure}

We are going to apply the theory we have developed here to the systems given by each Lagrangian. Consider the premultisymplectic forms $\Omega_{H_0}$ and $\Omega_{\bar H_0}$, where $H_0$ and $\bar H_0$ are the corresponding dynamical functions (equations \eqref{eq.hamiltonian} and \eqref{eq.hamiltonian.form}). Let $\hh$ and $\bar\hh$ denote solutions of the respective dynamical equations on $(W^1_0,\Omega_{H_0})$ and $(W^2_0,\Omega_{\bar H_0})$. They would locally have the form
\begin{eqnarray*}
\hh &=& \prth{\frac{\partial}{\partial x^j}+A_j\frac{\partial}{\partial u}+A_{ij}\frac{\partial}{\partial u_i}+B^i_j\frac{\partial}{\partial p^i}+C_j\frac{\partial}{\partial p}}\tensor\dd x^j,\\
\bar\hh &=& \prth{\frac{\partial}{\partial x^j}+\bar A_j\frac{\partial}{\partial u}+\bar A_{ij}\frac{\partial}{\partial u_i}+\bar A_{Kj}\frac{\partial}{\partial u_K}+\bar B^i_j\frac{\partial}{\partial p^i}+\bar B^{ki}_j\frac{\partial}{\partial p^{ki}}+\bar C_j\frac{\partial}{\partial p}}\tensor\dd x^j,
\end{eqnarray*}
where $|K|=2$. We then obtain the relations
\begin{eqnarray}
B^j_j &=& \frac{\partial L}{\partial u},\label{eq.1st.order.trace}\\
  p^i &=& \frac{\partial L}{\partial u_i},\label{eq.1st.order.w11}\\
  A_i &=& u_i,
\end{eqnarray}
for $(W^1_0,\Omega_{H_0},\hh)$; and
\begin{eqnarray}
\bar B^j_j &=& \frac{\partial L}{\partial u},\\
                p^i &=& \frac{\partial L}{\partial u_i} - \bar B^{ij}_j,\\
    p^{ij}+p^{ji} &=& (1_i+1_j)!\cdot\frac{\partial \bar L}{\partial u_{1_i+1_j}} = 0,\label{eq.1st.order.w21}\\
           \bar A_i &=& u_i,\\
        \bar A_{ij} &=& u_{1_i+1_j},
\end{eqnarray}
for $(W^2_0,\Omega_{\bar H_0},\bar\hh)$. Equations \eqref{eq.1st.order.w11} and \eqref{eq.1st.order.w21}, together with $H_0=0$ and $\bar H_0=0$, define the corresponding submanifolds $W^1_2$ and $W^2_2$ of $W^1_0$ and $W^2_0$.

We notice that, even though $\bar L$ is in some sense the same Lagrangian than $L$, a solution of the dynamical equation on $W^1_0$ may be easily determined, while in $W^2_0$ the space of solutions has grown (there are more coefficients to be determined). We thus infer from here, that a solution $\bar\hh$ of the dynamical equation in $W^2_0$ must satisfy an extra condition. Since $p=L-p^iu_i+0$ in $W^2_2$, the projection $\pi^{2,1}_0$ maps $W^2_2$ to $W^1_2$. We therefore impose to a solution $\bar\hh$ of the dynamical equation along $W^2_2$ to be in addition projectable to a solution $\hh$ of the dynamical equation along $W^1_2$. In such a case, we would have that
\begin{equation}
\bar B^{ij}_j=0
\end{equation}
which implies that the following equation
\begin{equation}
p^i=\frac{\partial L}{\partial u_i}
\end{equation}
is now a restriction in $W^2_2$. So, by tangency condition, we get
\begin{equation}
\bar B^i_j = \frac{\partial^2 L}{\partial x^j\partial u_i} + u_j\frac{\partial^2 L}{\partial u\partial u_i} + u_{1_k+1_j}\frac{\partial^2 L}{\partial u_k\partial u_i} + 0 = \frac{\dd}{\dx^j}\frac{\partial L}{\partial u_i}.
\end{equation}
Combining this with equation \eqref{eq.1st.order.trace}, we finally obtain
\begin{equation}
\frac{\partial L}{\partial u} - \frac{\dd}{\dx^j}\frac{\partial L}{\partial u_j} = 0,
\end{equation}
which is the Euler-Lagrange equation.

It is worth to remark here that, at this time, the Euler-Lagrange equation has not been deduced by the process shown in the proof of Proposition \ref{th.euler-lagrange.equations}, but directly from the projectability condition, although the previous Euler-Lagrange equation may be recovered from any of the two settings.
\end{example}

\section{Conclusion}
We have developed an intrinsic and global expression for the Euler-Lagrange equations for higher-order field theories. The main ingredients of this setting are the mixed space of velocities and momenta $W_0=J^k\pi\times_{J^{k-1}\pi}\Lambda^m_2(J^{k-1}\pi)$ and the premultisymplectic form
\[ \Omega_{H_0} = \Omega + \dd H_0\wedge\eta \]
defined on it, which encodes the dynamics of the system through the dynamical equation
\[ i_\hh\Omega_{H_0}=(m-1)\Omega_{H_0}. \]
We have analyzed in detail the existence of solution of this equation. Our approach  gives rise  to an unambiguous formulation of Lagrangian field theories of higher order.

In a future paper we will explore the extension of our techniques to the case of higher-order field theories with constraints, optimal control problems for partial differential equations and the implementation of numerical methods obtained directly from our approximation.

\appendix

\section{Connections} \label{sec.app.connections}
A connection $\Gamma$ in a fibration $\pi_{P,M}:P\To M$ is given by a horizontal distribution $\HH$ which is complementary to the vertical one $V\pi_{P,M}$, that is
\[ TP = \HH \oplus V\pi_{P,M}. \]
Associated to the connection there exists a horizontal projector $\hh:TP\To\HH$ defined in the obvious manner. If $(x^i,y^a)$ are fibered coordinates on $P$, then $\HH$ is locally spanned by the local vector fields
\[ \prth{\frac{\partial}{\partial x^i}}^h = \frac{\partial}{\partial x^i} + \Gamma^a_i(x,y)\frac{\partial}{\partial y^a} \;; \]
$(\partial/\partial x^i)^h$ is called the horizontal lift of $\partial/\partial x^i$, and $\Gamma^a_i$ are the Christoffel components of the connection.

Assume that $\pi_{P',M}:P'\To M$ and $\pi_{P,M}:P \To M$ are two fibrations with the same base manifold $M$, and that $\Upsilon:P'\To P$ is a surjective submersion (in other words, a fibration as well) preserving the fibrations, say, $\pi_{P,M}\circ\Upsilon=\pi_{P',M}$ (Diagram \ref{fig.preserved.fibration}). Let $\Gamma'$ be a connection in $\pi_{P',M}:P'\To M$ with horizontal projector $\hh$.

\begin{figure}[h]
\[\xymatrix{
  P' \ar[r]^\Upsilon \ar[dr]_{\pi_{P',M}} & P \ar[d]^{\pi_{P,M}} \\
                                          & M
  }\]
\caption{Preserved fibration} \label{fig.preserved.fibration}
\end{figure}

\begin{definition}
$\Gamma'$ is said to be projectable if the subspaces $T\Upsilon(z') (\HH_{z'})$ do not depend on $z'\in \Upsilon^{-1}(\Upsilon(z'))$.
\end{definition}

If $\Gamma'$ is projectable, then we define a connection $\Gamma$ in the fibration $\pi_{P,M}:P\To M$ as follows: The horizontal subspace at $z\in P$ is given by
\[ \bar{\HH}_z = T\Upsilon(z') (\HH_{z'}) \;, \]
for an arbitrary $z'$ in the fibre of $\Upsilon$ over $z$. It is routine to prove that $\bar{\HH}$ defines a horizontal distribution in the fibration $\pi_{P,M}:P\To M$.

We can choose fibered coordinates $(x^i,y^a,z^\alpha)$ on $P'$ such that $(x^i,y^a)$ are fibered coordinates on $P$. The Christoffel components of $\Gamma'$ are obtained by computing the horizontal lift
\[ \prth{\frac{\partial}{\partial x^i}}^h =
\frac{\partial}{\partial x^i} + \Gamma^a_i(x,y,z)\frac{\partial}{\partial y^a} + \Gamma^\alpha_i (x,y,z)\frac{\partial}{\partial z^{\alpha}} \;. \]
A simple computation shows that $\Gamma'$ is projectable if and only if the Christoffel components $\Gamma^a_i$ are constant along the fibres of $\Upsilon$, say $\Gamma^a_i=\Gamma^a_i(x,y)$. In this case, the horizontal lift of $\partial/\partial x^i$ with respect to $\Gamma$ is just
\[ \prth{\frac{\partial}{\partial x^i}}^h = \frac{\partial}{\partial x^i} + \Gamma^a_i(x,y)\frac{\partial}{\partial y^a} \;. \]
As an exercise, the reader can easily check that, conversely, given a connection $\Gamma$ in the fibration $\pi_{P,M}:P\To M$ and a surjective submersion $\Upsilon:P'\To P$ preserving the fibrations, one can construct a connection $\Gamma'$ in the fibration $\pi_{P',M}:P'\To M$ which projects onto $\Gamma$.

The notion of connection in a fibration admits a useful generalization to submanifolds of the total space. Let $\pi_{P,M}:P\To M$ be a fibration and $N$ a submanifold of $P$.

\begin{definition}
A connection in $\pi_{P,M}:P\To M$ along the submanifold $N$ consists of a family of linear mappings
\[ \hh_z:T_zP\To T_zN \]
for all $z \in N$, satisfying the following properties
\[ \hh_z^2=\hh_z, \quad \ker\hh_z= V\pi_{P,M})_z, \]
for all $z\in N$. The connection is said to be {\rm differentiable (flat)} if the distribution $\mathop{\mathrm{im}}\hh\subset TN$ is smooth (integrable).
\end{definition}

We have the following.

\begin{proposition}
Let $\hh$ be a connection in $\pi_{P,M}:P\To M$ along a submanifold $N$ of $P$. Then:
\begin{enumerate}
\item $\pi_{P,M}(N)$ is an open subset of $M$.
\item $(\pi_{P,M})_{|_N}:N\To\pi_{P,M}(N)$ is a fibration.
\item The 1-jet prolongation $J^1(\pi_{P,M})_{|_N}$ is a submanifold of $J^1\pi_{P,M}$.
\item There exists an induced true connection $\Gamma_N$ in the fibration $(\pi_{PM})_{|_N}:N\To\pi_{PM}(N)$ with the same horizontal subspaces.
\item $\Gamma_N$ is flat if and only if $\hh$ is flat.
\end{enumerate}
\end{proposition}

\begin{proof}
See \cite{LeonMrnMrr96,LeonMrrMrt03}.
\end{proof}

\section{Multi-index properties} \label{sec.app.multi-index.properties}
This section is devoted to some simple, but useful, properties of multi-indexes.

\begin{lemma} \label{th.multi-index.fubini}
Let $\set{a_{I,i}}_{I,i}$ be a family of real numbers indexed by a multi-index $I\in\NN^m$ and by an integer $i$ such that $1\leq i\leq m$. Given an integer $l\geq1$, we have that
\begin{equation} \label{eq.multi-index.fubini}
\sum_{|I|=l-1} \sum_{i=1}^m a_{I,i} = \sum_{|J|=l} \sum_{I+1_i=J} a_{I,i}.
\end{equation}
\end{lemma}

\begin{proof}
The proof is trivial when we realize that the sets $\set{(I,i)\ :\ |I|=l-1, 1\leq i\leq m}$ and $\set{(I,i)\ :\ I+1_i=J, |J|=l}$ are in bijective correspondence.
\end{proof}

\begin{lemma} \label{th.multi-index.identity}
Let $J\in\NN^m$ be a fixed multi-index. We have that
\begin{equation} \label{eq.multi-index.identity}
\sum_{I+1_i=J} \frac{I(i)+1}{|I|+1} = 1.
\end{equation}
\end{lemma}

\begin{proof}
\[ 1 = \sum_{i=1}^m \frac{J(i)}{|J|} = \sum_{I+1_i=J} \frac{J(i)}{|J|} = \sum_{I+1_i=J} \frac{I(i)+1}{|I|+1} \]
\end{proof}

\begin{lemma} \label{th.multi-index.lower.sum}
Let $\set{a_J,b^J}_J$ be a family of real numbers indexed by a multi-index $J\in\NN^m$. Given an integer $l\geq1$, we have that
\begin{equation} \label{eq.multi-index.lower.sum}
\sum_{|J|=l} b^Ja_J = \sum_{|I|=l-1} \sum_{i=1}^m \frac{I(i)+1}{|I|+1}(b^{I+1_i}+Q^{I,i})a_{I+1_i},
\end{equation}
where $\set{Q^{I,i}}_{I,i}$ is a family of real numbers such that for any multi-index $J\in\NN^m$ (with $|J|\geq1$) we have that
\begin{equation} \label{eq.multi-index.qs}
\sum_{I+1_i=J} \frac{I(i)+1}{|I|+1}Q^{I,i}=0.
\end{equation}
\end{lemma}

\begin{proof}
\begin{eqnarray*}
\sum_{|J|=l} b^Ja_J &=& \sum_{|J|=l} \prth{\sum_{I+1_i=J}  \frac{I(i)+1}{|I|+1}}b^Ja_J\\
                    &=& \sum_{|J|=l} \sum_{I+1_i=J}  \frac{I(i)+1}{|I|+1} (b^{I+1_i}+Q^{I,i})a_{I+1_i}.
\end{eqnarray*}
\end{proof}

%

\section*{Acknowledgements}
This work has been partially supported by the MICINN, Ministerio de Ciencia e Innovaci\'on (Spain), project MTM2007-62478, project ``Ingenio Mathematica'' (i-MATH) No. CSD 2006-00032 (Consolider-Ingenio 2010) and S-0505/ESP/0158 of the CAM (SIMUMAT). The first author (C. C.) also acknowledges the MICINN for an FPI grant.

J. V. is a Postdoctoral Fellow from the Research Foundation -- Flanders (FWO-Vlaan\-de\-ren), and a Fulbright Research Scholar at the California Institute of Technology. Additional financial support from the Fonds Professor Wuytack is gratefully acknowledged.

\bibliographystyle{siam}

\end{document}